\documentclass[reqno,11pt,centertags]{article}
\usepackage{amsmath,amsthm,amscd,amssymb,latexsym,upref}
\date{\today}                            

\input epsf
\usepackage{epsfig}
\usepackage[T2A,OT1]{fontenc}
\usepackage[ot2enc]{inputenc}
\usepackage[russian,english]{babel}
\usepackage[margin=3.5 cm]{geometry}
\usepackage{epic,eepicemu}
\usepackage[all]{xy}
\usepackage{enumitem}
\usepackage[center]{caption}

\usepackage{xcolor}

\usepackage{hyperref}

\newcommand{\bbN}{{\mathbb{N}}}
\newcommand{\bbR}{{\mathbb{R}}}

\newcommand{\bbC}{{\mathbb{C}}}

\DeclareMathAlphabet{\mathpzc}{OT1}{pzc}{m}{it}

\newcommand{\cL}{{\mathcal{L}}}

\newcommand{\cP}{{\mathcal{P}}}

\newcommand{\e}{{\epsilon}}

\renewcommand{\k}{\varkappa}

\newcommand{\pd}{{\partial}}

\def\restr#1{\,\vrule\,\lower1ex\hbox{$#1$}}

\def\a{\alpha}
\def\b{\beta}
\def\d{\delta}

\def\e{\epsilon}

\def\k{\kappa}

\def\O{\Omega}

\def\s{\sigma}

\def\t{\theta}

\newcommand{\dd}{\mathrm{d}}

\renewcommand{\Re}{\text{\rm Re}\,}
\renewcommand{\Im}{\text{\rm Im}\,}

\newcommand{\const}{\text{\rm const}}

\allowdisplaybreaks \numberwithin{equation}{section}
\newtheorem{theorem}{Theorem}[section]

\newtheorem{lemma}[theorem]{Lemma}
\newtheorem{proposition}[theorem]{Proposition}

\newtheorem{corollary}[theorem]{Corollary}

\theoremstyle{definition}
\newtheorem{definition}[theorem]{Definition}

\newtheorem{remark}[theorem]{Remark}
\newtheorem{problem}[theorem]{Problem}

\newtheorem{example}[theorem]{Example}

\date{\today}

\author{B.~Eichinger\thanks{Supported by the Austrian Science Fund FWF, project no: J4138-N32}\ \ \ and \ \   P.~Yuditskii\thanks{Supported by the Austrian Science Fund FWF, project no: P32855-N.}}

\title{Pointwise Remez inequality}
\begin{document}
	\maketitle
	
	\begin{center}
		\textit{Dedicated  A. Aptekarev on the occasion of his 65-th birthday}
	\end{center}

\begin{abstract}
	The standard well-known Remez inequality gives an upper estimate of the values of polynomials on $[-1,1]$ if they are bounded by $1$ on a subset of $[-1,1]$ of fixed Lebesgue measure. The extremal solution is given by the rescaled Chebyshev polynomials for one interval. Andrievskii asked about the maximal value of polynomials at a fixed point, if they are again bounded $1$ on a set of fixed size. We show that the extremal polynomials are either Chebyshev (one interval) or Akhiezer polynomials (two intervals) and prove Totik-Widom bounds for the extremal value, thereby providing a complete asymptotic solution to the Andrievskii problem.
\end{abstract}

\section{Introduction}

Based on several results by T.~Erd\'{e}lyi, E.~B. Saff and himself \cite{Anii1, Anii2, Erdelyi92, ErdelyiLiSaff94}, V. Andrivskii posed the following problem.

\begin{problem}\label{prob:ExtremalFixedSet}
	Let $\cP_n$ be the collection of polynomials of degree at most $n$.
	Let $E$ be a closed subset of $[-1,1]$, and $|E|$ denote its Lebesgue measure.  For $x_0\in[-1,1]$ define 
	\begin{equation}\label{eq:upperEnvSet}
	M_n(x_0,E)=\sup\{|P_n(x_0)|: P_n\in \cP_n,  |P_n(x)|\leq 1 \ \text{for}\ x\in E\}. 
	\end{equation}
	For $\delta \in (0,1)$ find
	\begin{equation}\label{eq:upperEnvelope}
	L_{n,\delta}(x_0)=\sup\{M_n(x_0,E): E\ \text{such that}\ |E|\geq2-2\delta\}. 
	\end{equation}

\end{problem}
Let us comment the setting with the following three evident remarks:
\begin{enumerate}
	\item $L_{n,\delta}(x_0)$ is even, thus we will consider only $x_0\in[-1,0]$.
	\item $\cL_{n,\delta}:=\sup_{x_0\in[-1,0]}L_{n,\delta}(x_0)$ is the famous Remez constant  \cite{Remes36}. It is attained at the endpoint $-1$
	by the the Chebyshev polynomial $ R_{n,\d}(x)$ for the interval $[-1+2\delta,1]$, i.e.,
	\begin{align*}
	\cL_{n,\delta}=\mathfrak{T}_n\left(\frac{1+\d}{1-\d}\right),\quad R_{n,\d}(x)=\mathfrak{T}_n\left(\frac{\d- x}{1-\d}\right),
	\end{align*}
	where
	$\mathfrak{T}_n$ denotes the standard Chebyshev polynomial of degree $n$ associated to $[-1,1]$,
	\begin{align*}
	\mathfrak{T}_n(x)=\frac{1}{2}\left((x+\sqrt{x^2-1})^n+(x-\sqrt{x^2-1})^n\right).
	\end{align*}
	We will henceforth call $R_{n,\d}$ the Remez polynomial. 
	\item Clearly, for Problem \ref{eq:upperEnvSet}, $ R_{n,\d}(x)$ can not be extremal for all $x_0\in[-1,0]$ as soon as $\d$ is sufficiently small. Let $\d=1/2$. Then 
	$R_{n,1/2}(0)=1$, while for the so called Akhiezer polynomial $A_{2m,\d}(x)$ we get
	$A_{2m,1/2}(0)=\mathfrak{T}_m\left(5/3\right)>1$. Recall \cite{Akh1} or \cite[Chapter 10]{AkhElliptic}
	$$
	A_{2m,\delta}(x)=\mathfrak{T}_m\left(\frac{1+\d^2-2x^2}{1-\d^2}\right)
	$$
	is the even Chebyshev polynomial on the set $[-1,1]\setminus(-\d,\d)$.
\end{enumerate}

Andrievskii raised his question on several international conferences, including Jaen Conference on Approximation Theory, 2018. The third remark highlights that the problem is non trivial. We found it highly interesting and in this paper we provide its complete asymptotic solution. 

In \cite{Akh1, Akh2} Akhiezer studied various extremal problems for polynomials bounded on two disjoint intervals $E(\a,\d)=[-1,1]\setminus(\a-\d,\a+\d)$, $\d-1<\a\le 0$, see also \cite[Appendix Section 36 and Section 38 in German translation]{Akh3}.

\begin{definition}\label{def:AkhiezerPol}
	We say that a polynomial $A_{n,\a,\delta}(x)$ is the Akhiezer polynomial in $E(\a,\d)$ with respect to  the {\em internal} gap if it solves the following extremal problem 
	\begin{equation}\label{28may1}
	A_{n,\a,\delta}(x_0)=\sup\{|P_n(x_0)|: \ P_n\in\cP_n,\ |P_n(x)|\le 1\ \text{for}\ x\in E(\a,\d)\},
	\end{equation}
	where $x_0\in(\a-\d,\a+\d)$.
\end{definition}

We describe properties of Akhiezer polynomials in Section \ref{sec:CombRep}. Note, in particular, that the extremal property of $A_{n,\a,\d}(x)$ does not depend on which point $x_0\in(\a-\d,\a+\d)$ was fixed in \eqref{28may1}.

In Section \ref{sec:Reduction} we prove the following theorem. 
\begin{theorem}\label{thm:AkhRemez}
	The extremal value $L_{n,\d}(x_0)$ is assumed either on the Remez polynomial $R_{n,\d}(x)$ or on an Akhiezer polynomial
	$A_{n,\a,\delta}(x)$ with a suitable $\a$.
\end{theorem}

Our main result, the asymptotic behaviour of $L_{n,\d}(x_0)$, is presented in Section \ref{sec:Asymptotics}.
Akhiezer provided asymptotics for his polynomials associated to two intervals in terms of elliptic functions (although the concept of the \textit{complex Green function} is used already in \cite{Akh1}). Nowadays the language of potential theory is so common and widely accepted, see  e.g. \cite{Anii2, ChriSiZiYuDuke,EichJAT17,KNT, Widom}, that we will formulate our asymptotic  result using this terminology.

Let $G_{\a,\d}(z,z_0)$ be the Green function in the domain $\Omega:=\overline\bbC\setminus E(\a,\d)$ with respect to $z_0\in\Omega$. In particular, we set $G_{\a,\d}(z)=G_{\a,\d}(z,\infty)$. Respectively, $G_{\d}(z)$ is the Green function of the domain 
$\overline\bbC\setminus E(\d)$ with respect to infinity, where $E(\d)=[-1+2\d,1]$. It is well known, that
\begin{align}\label{eq:RemezRep}
R_{n,\d}(x)=\frac{e^{n G_{\d}(x)}+e^{-n G_{\d}(x)}} 2,\quad x\in[-1, -1+2\d],
\end{align}
and
\begin{align}\label{eq:GreenBoundary}
\lim_{\a\to -1+\d}G_{\a,\d}(x)=G_{\d}(x),\quad x\in(-1, -1+2\d).
\end{align}

\begin{lemma}\label{lem:UpperEnvelopeGreen}
	Let $\Phi_\d(x)$ be the upper envelope
	\begin{align}\label{eq:upperEnvelopeGreen}
		\Phi_\d(x)=\sup_{\a\in (\d-1,0]}G_{\a,\d}(x).
	\end{align}
	If for $x_0$ the supremum is attained for some internal point $\a\in (\d-1,0]$, then $x_0=x_0(\a)$ is a unique solution of the equation
	$$
	\pd_\a G_{\a,\d}(x)=0.
	$$
	Otherwise, it is attained as $\a\to \d-1$ and $\Phi_\d(x_0)=G_{\d}(x_0)$.
\end{lemma}
Note that the only nontrivial claim in Lemma \ref{lem:UpperEnvelopeGreen} is the uniqueness of $x_0(\a)$. We will provide a proof of this fact in Section \ref{sec:AsympDiagram} und use it to give a description of the upper envelope $\Phi_\d(x_0)$ in Proposition \ref{prop:UpperEnvelope}.

Our main theorem below shows that the asymptotics of $L_{n,\d}(x_0)$ are described by the upper envelope $\Phi_\d(x_0)$.
\begin{theorem}\label{thm:TotikWidomBounds}
	The following limit exists
	$$
	\lim_{n\to\infty}\frac 1 n\log 2L_{n,\d}(x_0)=\Phi_\d(x_0).
	$$
	To be more precise, if $\Phi_\d(x_0)=G_{\d}(x_0)$, then 
	\begin{align}\label{eq:TotikWidomRem}
	\log 2L_{n,\d}(x_0)=n \Phi_\d(x_0)+o(1).
	\end{align}
	If $\Phi_\d(x_0)=G_{\d,\a}(x_0)$ for some $\a\in(\d-1,0]$,  then the Totik--Widom bound 
	\begin{align}\label{eq:TotikWidomAkh}
		\log 2L_{n,\d}(x_0)=n \Phi_\d(x_0)+O(1)
	\end{align}
	holds.
\end{theorem}

\section{Preliminaries}

In \cite{TikhYud},  the sharp constant in the Remez inequality for trigonometric polynomials on the unit circle was given. The proof was based on the following two steps: 
\begin{itemize}
	\item[(i)] the structure of possible extremal polynomials was revealed with the help of their comb representations,
	\item[(ii)] the principle of harmonic measure (a monotonic dependence on a domain extension) allows 
	to get an extremal configuration for the comb parameters in the given problem.
\end{itemize}

We also start with recalling comb representations for extremal polynomials, see Subsection \ref{sec:CombRep}. We refer the reader to \cite{ErY12, SodYud92} and the references therein for more information about the use of comb mappings in Approximation theory. However, we doubt that the step (ii) is applicable to the Andievskii problem. That is, that comparably simple arguments from potential theory such as the principle of harmonic measure, would also allow  us to bring a certain fixed configuration (comb) to an extremal one. Instead, we develop here an \textit{infinitesimal approach}, closely related to the ideas of Loewner chains \cite{ConwayCompAna,PomUnivalent}.

Using this method, we will prove in Section \ref{sec:Reduction} that the extremal solution for $L_{n,\d}(x_0)$ is either a Chebyshev
or an Akhiezer polynomial. For this reason, we continue  in Subsection \ref{sec:Akhiezer} with the complete discussion of
Akhiezer polynomials $A_{n,\a,\d}$, see \eqref{28may1}. Recall  that $A_{n,\a,\d}$ is extremal on two given intervals with respect to
points $x_0 \in (\a-\d,\a+\d)$. Up to some trivial degenerations, it is different to the classical Chebyshev polynomial related to the same set, since the latter one is extremal for points $x_0\in \bbR \setminus [-1, 1]$. Generically,  $A^{-1}_{n,\a,\d}([-1, 1])$ is a union of \textit{three} intervals --  
the set  contains also an additional interval outside of $[-1, 1]$. We demonstrate our infinitesimal approach showing dependence of this additional interval on $\a$ for fixed $n$ and $\delta$. In the language of comb domains, we will observe a rather involved dependence of the comb parameters in a simple monotonic motion of $\a$. The domain can undergo all possible variations: to narrow, expand, or a combination of both, see Theorem \ref{thm:ystarH0}. Essentially, this is  the base for our believe that simple arguments, in the spirit of  (ii), in the Andrievskii problem are hardly possible.  

If for fixed $x_0$ the extremal polynomial is a Remez polynomial, there is no additional interval outside of $[-1,1]$. Intuitively, the same considerations as were used in \cite{TikhYud} should work. However, a technical difference prevents a direct applications of the principle of harmonic measure. Namely, the Lebesgue measure on the circle corresponds to the harmonic measure evaluated at $0$, while in the sense of potential theory the Lebesgue measure on $\bbR$ corresponds to the Martin or Phragm\'{e}n-Lindel\"of measure, \cite{KoosisLogInt1}. Although we are convinced that a limiting process would allow to overcome this technical issue, we provide an alternative proof below; cf Lemma \ref{lem:Remez}. 

\subsection{Comb representation for extremal polynomials}\label{sec:CombRep}
By a regular comb, we mean a half strip with a system of vertical slits
\begin{align*}
\Pi=\{z: \Re z\in (\pi n_-,\pi n_+), \ \Im z>0 \}\setminus\{z=\pi k+iy: y\in (0,h_k),  k\in (n_-,n_+)\}
\end{align*}
where $n_+-n_-=n\in\bbN$ and the $k$'s are integers. We call $h_k$, $h_k\geq 0$, the height of the $k$-th slit and point out that the degeneration $h_k=0$ is possible. Let $\t:\bbC_+\to\Pi$ be a conformal mapping of the upper half-plane onto a regular comb such that $\t(\infty)=\infty$. Then
\begin{equation}\label{28oct2}
T_n(z)=\cos \t(z)
\end{equation}
defines a polynomial of degree $n$. Let $E\subset [-1,1]$ be compact and $x_0\in[-1,1]\setminus E$. Moreover, let $(a,b)$ denote the maximal open interval in $\bbR\setminus E$ that contains $x_0$. Then using the technique of Markov correction terms we obtain the following representation of the extremizer of \eqref{eq:upperEnvSet}
\begin{theorem}[{\cite[7.5. Basic theorem]{SodYud92}, \cite[Theorem 3.2.]{EichYu18}}]\label{thm:extension}
	Under the normalization $T_n(x_0)>0$, there exists a unique extremal polynomial for \eqref{eq:upperEnvSet} and it only depends on the gap $(a,b)$ and not on the particular point $x_0\in(a,b)$. Moreover, let $\t:\bbC_+\to\Pi$ be a comb mapping onto a regular comb $\Pi$ and $E$ be such that:
	\begin{itemize}
		\item[(i)]  $h_0>0$ and $a=\t^{-1}(-0), b=\t^{-1}(+0)$,
		\item[(ii)] $E$ contains at least one of the points $\t^{-1}(k\pi\pm 0)$, for all $k\in(n_-,n_+)$,
		\item[(iii)] $E$ contains at least one of the points $\t^{-1}(\pi n_-), \t^{-1}(\pi n_+)$,
	\end{itemize} 	
	Then,  
	\begin{align}\label{eq:SpecialRepExtremalPoly}
	T_n(z)=\cos\t(z)
	\end{align}
	is an extremal polynomial for $E$ and the interval $(a,b)$. Vice versa, if $T_n$ is an extremal polynomial for a set $E$ and an interval $(a,b)$, then there exists a regular comb with these properties such that \eqref{eq:SpecialRepExtremalPoly} holds.
\end{theorem}

Let us now specify to the case of Akhiezer polynomials, where $E=E(\a,\d)=[-1,1]\setminus (\a-\d,\a+\d)$. By the above theorem, the so-called $n$-extension $E_n=E_n(\a,\d):=A_{n,\a,\d}^{-1}([-1,1])$ can be of the following types:
\begin{enumerate}[label=(\roman*)]
	\item there is an additional interval to the right of $1$,
	\item $E$ is extended at $1$,
	\item $E$ is extended at $-1$,
	\item there is an additional interval to the left of $-1$.
\end{enumerate} 
The corresponding comb-mapping realization is collected in the following Corollary:
\begin{corollary}\label{cor:AkhPoly}
	For a fixed set $E=E(\a,\d)$ and  degree $n$ the extremal polynomial $A_{n,\a,\d}(z)$ of \eqref{eq:upperEnvSet} is of the form \eqref{eq:SpecialRepExtremalPoly}. The unique comb $\Pi_{n,\a,\d}$ has one of four possible shapes shown in Figure \ref{fig:cases1and2} and \ref{fig:cases3and4}. Moreover, the following normalizations distinguish the cases 
	\begin{enumerate}[label=(\roman*)]
		\item\label{it:Combcase1} $\t^{-1}(\pi n_-)=-1$, $\t^{-1}(\pi (n_+-1)-0)=1$, see Figure \ref{fig:cases1and2} left,
		\item\label{it:Combcase2} $\t^{-1}(\pi n_--0)=-1$, $\t(1)\in [\pi (n_+-1), \pi n_+]$ see Figure \ref{fig:cases1and2} right,
		\item\label{it:Combcase3} $\t(-1)\in [\pi (n_-), \pi (n_-+1)]$, $\t^{-1}(\pi n_--0)=-1$, see Figure \ref{fig:cases3and4} left,
		\item\label{it:Combcase4} $\t^{-1}(\pi (n_-+1)+0)=-1$, $\t^{-1}(\pi n_+)=1$, see Figure \ref{fig:cases3and4} right.
	\end{enumerate}
\end{corollary}

\begin{remark}
Let us denote the additional interval for the cases \ref{it:Combcase1} and \ref{it:Combcase4} by $I_n$. These cases include the limit cases $h_{n_+-1}=\infty$ and $h_{n_-+1}=\infty$, respectively.  That is, the extremal polynomial is of the degree $n-1$. Note that $A_{n,\a,\d}$ has a zero on $I_n$. The aforementioned degeneration corresponds to ``the zero being at $\infty$''. On the other hand, also the degenerations $h_{n_+-1}=0$ and $h_{n_-+1}=0$ are possible, which allow for a  smooth transition to the  cases \ref{it:Combcase2} and \ref{it:Combcase3}, respectively.
\end{remark}

\subsection{Reduction to Remez polynomials}\label{sec:Remez}
As we have mentioned in the beginning of this Section, we cannot use the ideas of harmonic measure directly. We overcome this technical problem by using transition functions from the theory of Loewner chains \cite{ConwayCompAna,PomUnivalent}. Let us consider an arbitrary regular comb  $\Pi=\Pi(0)$ as in Theorem \ref{thm:extension} and let us fix a slit with height $h_k>0$. Let $h_k(\e)$ be strictly monotonically decreasing such that $h_k(0)=h_k$ and $\Pi(\e)$ be the comb which coincides with $\Pi$, but $h_k$ is reduced to $h_k(\e)$. Let $\t$ and $\t_\e$ be the corresponding comb mappings. Then, clearly $\Pi\subset\Pi_{\e}$ and thus the transition function 
\begin{align*}
w(z,\e)=w_\e(z):=\t^{-1}_\e(\t(z))
\end{align*}
is well defined and is an analytic map from $\bbC_+$ into $\bbC_+$. If we define $I(\e):=\t^{-1}([\pi k+ih_k(\e), \pi k+ih_k])$ and $c_k(\e)=\t_\e^{-1}(h_k(\e))$, then $w_\e:\bbR\setminus I(\e)\to \bbR\setminus\{c_k(\e)\}$ is one-to-one and onto. The arc $J(\e):=\t_\e^{-1}([\pi k+ih_k(\e), \pi k+ih_k])$ lies, except its endpoint in $\bbC_+$ and $w_\e:\bbC_+\to\bbC_+\setminus J(\e)$ is conformal.
\begin{lemma}\label{lem:decreasingVariation}
	The Nevanlinna function $w_\e$ admits the representation 
	\begin{equation}\label{28oct4}
	w(z,\e)=z+\int_{I(\e)}\frac{1-z^2}{1-\xi^2}\frac{\dd\sigma_\e(\xi)}{\xi-z}.
	\end{equation}
\end{lemma}
\begin{proof}
	It follows from the above, that $w_\e$ is a Nevanlinna function and that the measure in its integral representation is supported on $I(\e)$. That is, we can write 
	\begin{align*}
	w_\e(z)=A_\e z+B_\e+\int_{I(\e)}\frac{\dd\sigma_\e(\xi)}{\xi-z},\quad A_\e\geq 0,B_\e\in\bbR.
	\end{align*}
	Moreover, it satisfies the normalization conditions
	\begin{align}\label{eq:transitionNormalization}
	w_\e(-1)=-1,\quad w_\e(1)=1,\quad w_\e(\infty)=\infty,
	\end{align}
	which yields
	\begin{align}\label{eq:Aexplicit}
	B_\e=1-A_\e-\int\frac{\dd\sigma_\e(\xi)}{\xi-1},\quad A_\e=1+\int_{I(\e)}\frac{\dd\sigma_\e(\xi)}{1-\xi^2}.
	\end{align}
	From this, we obtain \eqref{28oct4}. 
\end{proof}

We are now ready to prove the easy part in Theorem \ref{thm:AkhRemez}. That is, if $x_0$ is in a gap on the boundary, then the extremal polynomial is actually the Remez polynomial. 
\begin{lemma}\label{lem:Remez}
	Let 
	\begin{align*}
	E=[b_0,1]\setminus \bigcup_{j=1}^g(a_j,b_j),\quad x_0\in[-1,b_0),\quad |E|=2-2\d
	\end{align*}
	and $T_n(x,E)$ be the extremal polynomial for $x_0$ as described in Theorem \ref{thm:extension}. Then, 
	\begin{align*}
	R_{n,\d}(x_0)> T_n(x_0,E).
	\end{align*}
\end{lemma}
\begin{proof}
	Let $\Pi$ be the comb associated to $E$ and let us assume that $n_-=0, n_+=n$. Noting that by Theorem \ref{thm:extension} there is no extension in the extremal gap, we obtain that $\t(x_0)\in i\bbR_+$. We will show that by decreasing the slits, the set as well as the value of the extremal polynomial will increase. Let us decrease the slit corresponding to the gap $(a_j,b_j)$. By Lemma \ref{lem:decreasingVariation} this is achieved by 
	\begin{align*}
	w_\e(z)=z+\int\frac{z^2-1}{\xi^ 2-1}\frac{\dd\sigma_\e(\xi)}{\xi-z}.
	\end{align*}
	Thus, a direct computation shows
	\begin{align*}
	w_\e(a_{l+1})-w_\e(b_{l})&=a_{l+1}-b_l+(a_{l+1}-b_l)\int \left(\frac{1}{1-\xi^2}+\frac{1}{(\xi-a_{l+1})(\xi-b_l)}\right)d\sigma_\e(\xi)
	\\&>a_{l+1}-b_l,
	\end{align*}
	where we used that the support of $\dd\sigma_\e$ is included in the gap $(a_j,b_j)$. This also holds for $l=g$, with $a_{g+1}=1$. Thus, the size of the bands $[b_l,a_{l+1}]$ increases and we get that $|E_\e|$ is increasing. Similarly, we see that $w_\e(x_0)>x_0$. 
	Recall that $w_\e(z)=\t_\e^{-1}(\t(z))$. Using that $\t_\e$ is decreasing on $i\bbR_+$, we get
	\begin{align*}
	\Im \t_\e(x_0)> \Im \t_\e(w_\e(x_0))=\Im\t(x_0).
	\end{align*}
	Thus, $T_\e(x):=\cos(\t_\e(x))$ satisfies $T_\e(x_0)>T(x_0)$. Therefore, by this procedure, we can show that there is $b_0'$ with $1-b_0'>|E|$ and $T_n(x_0,[b_0',1])>T_n(x_0,E)$. Set $2-2\d':=1- b_0'$ and note that $\d'<\d$. Since $R_{n,\d}(x_0)$ is clearly monotonic increasing in $\d$ it follows that $T_n(x_0,[b'_0,1])=R_{n,\d'}(x_0)\leq R_{n,\d}(x_0)$, which concludes the proof. 
\end{proof}
\begin{remark}
	Continuing the heuristics provided in the beginning of this section, we give an interpretation of this proof in terms of the harmonic measure. The argument in the proof relies on showing that $|E_\e|>|E|$ and $w_\e(x_0)>x_0$. Let us assume we were interested in the harmonic measure of $E$ instead of its Lebesgue measure. Then the above properties have a probabilistic interpretation: Namely, the first one corresponds to the probability that a particle which starts at a point which is close to infinity in the domain $\Pi_\e$ and $\Pi$ and terminates on the base of the comb. From this perspective it is clear that this increases if one of the slit heights is decreased. Similarly, writing the second one as $w_\e(x_0)+1>x_0+1$, it corresponds to the probability that a particle terminates in $[\t(-1),\t(x_0)]$ in $\Pi$ and $\Pi_\e$, respectively. Again, this explains, why the value should be increased.
\end{remark}
The above interpretation uses the fact that there is no additional portion of $E$ outside of $[-1,1]$. For Akhiezer polynomials this is not true. This makes the problem essentially more delicate and we extend our tools by using the infinitesimal variation $\dot w$ as defined in \eqref{def:dotW}; cf. also Remark \ref{rem:slitheights}.

\subsection{Transformation of the Akhiezer polynomials as $\a$ is moving}\label{sec:Akhiezer}
Our goal is to describe the transformation of the comb as the interval starting from the center moves continuously to the left. This should correspond to a continuous transformation of the comb, which seems at the first sight impossible, since the base of the slits are positioned only at integers. We show in the following how the cases \ref{it:Combcase1}, \ref{it:Combcase2}, \ref{it:Combcase3}, \ref{it:Combcase4} can be connected by a continuous transformation: Let us start with case \ref{it:Combcase1} in the degenerated configuration $h_{n_+-1}=\infty$. Then we start decreasing $h_{n_+-1}$ until we reach $h_{n_+-1}=0$. Now we are in the situation that $\t(1)=\pi(n_+-1)$ and we continue with case \ref{it:Combcase2}. That is, $\t(1)$ increases until it reaches $\t(1)=\pi n_+$. All the time $\t(-1)=\pi n_-$. Now we continue with case \ref{it:Combcase3} and increase $\t(-1)$ until the point $\t(-1)=\pi (n_-+1)$. We continue with case \ref{it:Combcase4} and increase $h_{n_-+1}$ from $h_{n_-+1}=0$ to $h_{n_-+1}=\infty$. We have arrived at our initial configuration but the base of the comb was shifted from position
$(\pi n_-,\pi(n_+-1))$ to position $(\pi (n_-+1),\pi n_+)$.

\begin{figure}[h]
	\includegraphics[scale=0.3]{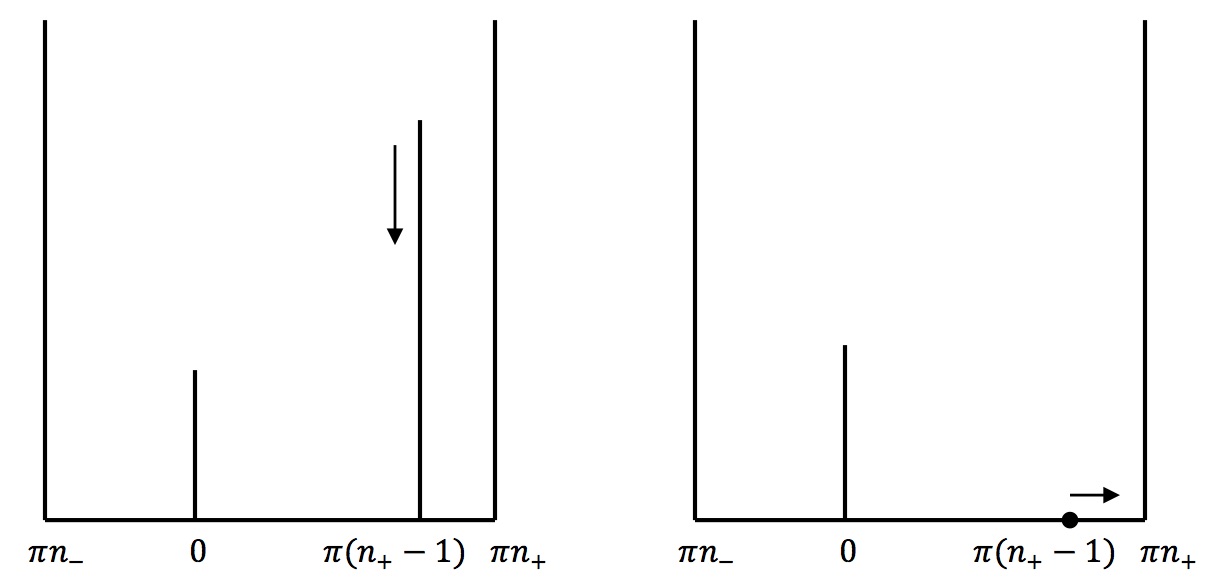}
	\caption{Comb domains for the cases (i) and (ii).}
	\label{fig:cases1and2}
\end{figure}
\begin{figure}[h]
	\includegraphics[scale=0.3]{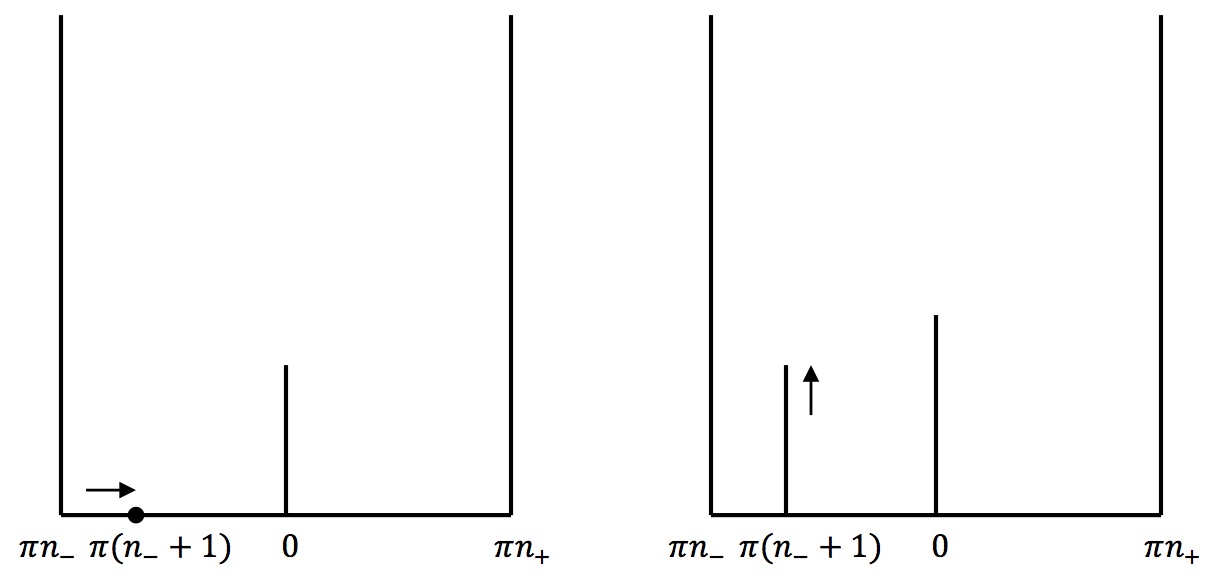}
	\caption{Comb domains for the cases (iii) and (iv).}.
	\label{fig:cases3and4}
\end{figure}

We believe it is helpful to understand these transformation also on the level of $E_n$. Our initial configuration corresponds to $E=E_n$. We view this as the additional interval is hidden at $\infty$. Starting case \ref{it:Combcase1} corresponds to the motion that the position of the addition interval decreases from $+\infty$ until it joins $E$. This corresponds to the change from case \ref{it:Combcase1} to \ref{it:Combcase2}. When it is fully absorbed, we arrive at case \ref{it:Combcase3}, i.e., $E$ starts to be extended to the left until the extension separates and start to decrease to $-\infty$. And the circle starts from the beginning. 
\begin{theorem}\label{thm:transformation}
	Moving $(a,b)$ to the left corresponds to a succession of continuous transformations of the comb as described above. If $n$ is odd, we start with case \ref{it:Combcase1} and $h_{n_+-1}=\infty$. If $n$ is even, we start with case \ref{it:Combcase3} and $\t(-1)=\pi n_-$. 
\end{theorem}
\begin{remark}
	When $\a=0$ by symmetry the extremal polynomial is even and all critical values are equally distributed on $[-1,-\d]$ and $[\d,1]$. The above procedure describes how all critical values are moved from the interval $[-1,\a-\d]$ to $[\a+\d,1]$ as $\a$ is decreasing. In the limit when $\a$ approaches $-1+\d$ all critical values will be on $[-1+2\d,1]$ and the extremal polynomial will transform into the Remez polynomial $R_{n,\d}$. 
\end{remark}

As we have already indicated at the end of Section \ref{sec:Remez} the proof of Theorem \ref{thm:transformation} requires in addition to the transition function $w_\e$ its infinitesimal transform $\dot w$. Let us introduce the notation 
\begin{align}\label{def:dotW}
w'(z,\e)=\partial_zw(z,\e),\quad\dot w(z,\e)=\partial_\e w(z,\e)\quad\text{and}\quad  \dot w(z)=\dot w(z,0).
\end{align}
\begin{lemma}
	Under an appropriate choice of $\e$, $w(z,\e)$ is differentiable with respect to $\e$ and there exists $\sigma_k>0$ such that 
	\begin{equation}\label{28oct5}
	\dot w(z):=\lim_{\e\to 0}\frac{w(z,\e)-w(z,0)}{\e}= \frac{1-z^2}{1-c_k^2}\frac{\s_k}{c_k-z},
	\end{equation}
	where $c_k=c_k(0)=\t^{-1}(\{k+ih_k\})$. 
\end{lemma}
\begin{proof}
	We note that if $k=n_-+1$ or $k=n_+-1$ and we are in the situation of Theorem \ref{thm:extension} \ref{it:Combcase1} or \ref{it:Combcase4}, then due to \eqref{eq:Aexplicit} $A_\e<1$ and $A_\e>1$ otherwise. In any case, we obtain from the properties of $w_\e$ that $A_\e$ is strictly monotonic. Therefore, if we fix $\e_0$ and define 
	$
	\b_0=\left|\int_{I(\e_0)}\frac{\dd\sigma_{\e_0}(\xi)}{1-\xi^2}\right|,
	$
	then 
	$
	\e\mapsto  \left|\int_{I(\e)}\frac{\dd\sigma_\e(\xi)}{1-\xi^2}\right|
	$
	maps $[0,\e_0]$ continuously and bijectively on $[0,\b_0]$. Through a reparametrization we can achieve that
	$
	\left|\int_{I(\e')}\frac{\dd\sigma_{\e'}(\xi)}{1-\xi^2}\right|=\frac{\e'\b_0}{\e_0}.
	$
	Thus, the measures $\frac{\dd\sigma_{\e'}(\xi)}{(1-\xi^2)\e'}$ are normalized and we get by passing to a subsequence 
\begin{align*}
	\lim_{k\to \infty}\frac{w(z,\e'_k)-z}{\e'_k}=\lim\limits_{k\to \infty}\int_{I(\e'_k)}\frac{1-z^2}{1-\xi^2}\frac{1}{\xi-z}\frac{d\sigma_{\e'_k}(\xi)}{\e_k'}=\int\frac{1-z^2}{1-\xi^2}\frac{\dd \sigma_\infty(\xi)}{\xi-z}.
\end{align*}
Since the support of $\sigma_{\e'_k}$ shrinks to $\{c_k\}$, we get that $\sigma_\infty=\sigma_k\delta_{c_k}$, for $\sigma_k/|1-c_k^2|=\b_0/\e_0$. This concludes the proof. 
\end{proof}
We are also interested in the case of growing slit heights. This corresponds to 
\begin{align*}
u_\e(x)=w_\e^{-1}(x),
\end{align*}
for some transition function as defined above. Note that since $c_k(\e)\to c_k$ as $\e\to 0$, this inverse is well defined on $\bbR$ away from a vicinity of $c_k$ and we conclude since $w'(z,0)=1$ that 
\begin{align}\label{eq:uderivative}
\dot u(x)=-\dot w(x).
\end{align}

The following lemma discusses the case \ref{it:Combcase1}. We show that decreasing $h_{n_+-1}$ and simultaneously increasing $h_0$ appropriately allows us to move the gap $(\a-\d,\a+\d)$ to the left without changing its size. Let us set $(a,b)=(\a-\d,\a+\d)$ let $c$ be the critical point in $(a,b)$ and $c_+$ be the critical point outside $[-1,1]$, i.e., $c=\t^{-1}(ih_0)$, $c_+=\t^{-1}(n_+-1+ih_{n_+-1})$.

\begin{lemma}\label{l29oct1}
	Let $E(\a,\d)$ with $\a\leq 0$ be such that the corresponding extremal polynomial $A_{n,\a,\d}(z)$ corresponds to the case \ref{it:Combcase1}. 
	Then the infinitesimal variation $\dot w(x)$ generated by decreasing the height  $h_{n_+-1}$ under the constraint of  a constant gap length ($\delta=\const$) leads to an increase of $h_0$. Moreover, in this variation the gap is moving to the left, that is, $\a(\e)$ is decreasing. 
\end{lemma}
\begin{proof}
	Let $w_{n_+-1}$ and $w_0$ be transforms corresponding to a variation of the slit heights $h_{n_+-1}$ and $h_0$. In \eqref{28oct5} we chose $\s_{n_+-1}>0$ and determine the sign of $\s_0$ by the condition of constant gap length. Thus, the total transform corresponds to $w(z,\e)=w_{n_+-1}(w_0(z,\e),\e)$ and hence, by the previous computations, we find that 
	\begin{align}\label{eq:30Jan1}
	\dot{w}(z)=\frac{1-x^2}{1-c^2}\frac{\s_0}{c-x}+\frac{1-x^2}{1-c_+^2}\frac{\s_+}{c_+-x},
	\end{align}
	with $\s_+>0$. The value $\s_0$ is determined due to 
	the constraint $w(b,\e)-w(a,\e)=b-a$, i.e., $\dot w(b)=\dot w(a)$. 
	Thus, we obtain 
	\begin{align}\nonumber
	0&=\left(\frac{1-b^2}{c-b}-\frac{1-a^2}{c-a}
	\right)\frac{\s_0}{1-c^2}+\left(\frac{1-b^2}{c_+-b}-\frac{1-a^2}{c_+-a}
	\right)\frac{\s_+}{1-c_+^2}
	\\ \label{29}
	&=(b-a)\left(\frac{1+ab-c(a+b)}{(c-a)(c-b)}\frac{\s_0}{1-c^2}+\frac{1+ab-c_+(a+b)}{(c_+-a)(c_+-b)}\frac{\s_+}{1-c_+^2}\right).
	\end{align}
	Let 
	\begin{align}\label{def:ell}
	\ell(x)=1+ab-(a+b)x.
	\end{align} 
	Since $a+b=2\a\leq 0$, $\ell$ is non-decreasing. Moreover, since $a,b>-1$,  $\ell(-1)=(1+a)(1+b)>0$ and hence, $\ell(c)>0$ and $\ell(c_+)>0$. Using that the numerator is negative for both summands, we find that $\s_+>0$ implies $\s_0<0$. In other words,  the compression of the height $h_{n_+-1}$ leads to a growth of $h_0$. Since both summands are negative,
	$$
	\dot w(a)=\frac{1-a^2}{1-c^2}\frac{\s_0}{c-a}+\frac{1-b^2}{1-c_+^2}\frac{\s_+}{c_+-a}<0.
	$$
	Consequently, $\dot w(b)=\dot w(a)<0$, and we find that the ends of the interval $(a_\e,b_\e)$ are moving to the left. This concludes the proof. 
\end{proof}

Case \ref{it:Combcase4} is similar to case \ref{it:Combcase1}. However, we will encounter a certain specific phenomena. First of all, we will increase the value of $h_{n_-+1}$ to move the given gap to the left. But in order to fulfill the constraint of fixed gap length both an increasing or a decreasing of $h_0$ is possible.
\begin{lemma}\label{l29oct2}
	Let $E(\a,\d)$ with $\a\leq 0$ be such that the corresponding extremal polynomial $A_{n,\a,\d}(z)$ corresponds to the case \ref{it:Combcase4}. 
	Let $\ell$ be defined as in \eqref{def:ell} and 
	\begin{align}\label{def:eta}
	\eta:=\frac 1 2\left(\a+\frac{1-\d^2}\a\right), \quad (\eta<-1),
	\end{align}
	be its zero. If $c_-<\eta$,  the infinitesimal variation $\dot w(x)$ generated by increasing the height  $h_{n_-+1}$ under the constraint of  a constant gap length ($\delta=\const$) leads to an increase of $h_0$ and it leads to a decrease of $h_0$ if $c_-\in (\eta,-1)$. In both case, the gap is moving to the left, that is, $\a(\e)$ is decreasing. 
\end{lemma}
\begin{proof}
	As before, we find that the infinitesimal variation is of the form
	\begin{equation*}
	\dot w(x)=\frac{1-x^2}{1-c^2}\frac{\s_0}{c-x}+\frac{1-x^2}{1-c_-^2}\frac{\s_-}{c_--x},\quad \s_-<0.
	\end{equation*}
	Respectively, the constraint $\dot w(b)=\dot w(a)$ corresponds to
	\begin{align*}
	\frac{\ell(c)}{(c-a)(c-b)}\frac{\s_0}{1-c^2}+\frac{\ell(c_-)}{(c_--a)(c_--b)}\frac{\s_-}{1-c_-^2}=0.
	\end{align*}
	We have that $\ell(c_-)<0$ for $c_-<\eta$. Since $c_-<-1<a<c<b$, this implies that $\s_0<0$. As before, we conclude that 
	$
	\dot w(a)<0,
	$
	and the interval is moving to the left. If $c_-\in(\eta,-1)$, then $\ell (c_-)>0$ and therefore $\s_0>0$. In this case, 
	$
	\dot w(b)<0
	$
	and again the interval is moving to the left. Finally, if $c_-=\eta$, we have $\ell(c_-)=0$ and therefore \eqref{29} implies that $\s_0=0$.
\end{proof}
\subsubsection{The cases (ii) and (iii)}
We will discuss case \ref{it:Combcase2} and case \ref{it:Combcase3} simultaneously. Recall that case \ref{it:Combcase2} corresponds to an extension of $E$ to the right and case \ref{it:Combcase3} to an extension to the left, i.e., $\t(1)\in (\pi(n_+-1),\pi n_+)$ or  $\t(-1)\in (\pi n_-,\pi( n_-+1))$. Let $\Pi\equiv\Pi(\e)$ but let us increase the normalization point $\t_\e(\pm 1)$. Let $w^\pm_\e(z)=\t_\e^{-1}(\t(z))$ be the corresponding transition functions.
\begin{lemma}\label{lem:variationTranslation}
	Let $w^\pm_\e$ be defined as above. Then there exists $\rho^\pm_\e>0$, such that 
	\begin{align}\label{eq:1June1}
	w^+_\e(z)=-1+\rho^+_\e(z+1),\quad w^-_\e(z)=1+\rho^-_\e(z-1),\quad \rho^+_\e<1,\rho_\e^->1
	\end{align} 
	and 
	\begin{align}\label{1jul4}
	\dot w^\pm(z)=\tau^\pm(1\pm z),\quad \tau^\pm<0.
	\end{align}
\end{lemma}
\begin{proof}
	We only prove the claim for $w^+_\e$. Since $\Pi(\e)\equiv \Pi(0)$, $w_\e$ is just an affine transformation and using $w_\e(-1)=-1$ and $w_\e(\infty)=\infty$ we find \eqref{eq:1June1}. Since $\t_\e(1)>\t(1)$, we obtain that $w_\e(1)=\t_\e^{-1}(\t(1))<\t_\e^{-1}(\t_\e(1))=1$ and thus $\rho_\e<1$. Therefore,
	\begin{align*}
	\dot w(z)=\lim\limits_{\e\to\infty}\frac{w_\e(z)-z}{\e}=(z+1)\lim\limits_{\e\to\infty}\frac{\rho_\e-1}{\e}=\tau<0.
	\end{align*}
\end{proof}
\begin{lemma}
	Let $E(\a,\d)$ with $\a\leq 0$ be such that the corresponding extremal polynomial $A_{n,\a,\d}(z)$ corresponds to the case \ref{it:Combcase2} (case \ref{it:Combcase3}). Then the infinitesimal variation $\dot w(x)$ generated by increasing $\t(1)$ (increasing $\t(-1)$) under the constraint of  a constant gap length ($\delta=\const$) leads to an increase (decrease) of $h_0$. Moreover, in this variation the gap is moving to the left, that is, $\a(\e)$ is decreasing. 
\end{lemma}
\begin{proof}
	We only prove the case \ref{it:Combcase2}.
	We have
	\begin{align}\label{1jul5}
	\dot w(x)=\tau (x+1)+\frac{1-x^2}{1-c^2}\frac{\s_0}{c-x}, \quad \tau<0.
	\end{align}
	 Applying the constraint $\dot w(a)=\dot w(b)$, we obtain
	$$
	0=(b-a)\left(\tau+\frac{1+ab-(a+b)c}{(c-a)(c-b)}\frac{\s_0}{1-c^2}.
	\right)
	$$
	Since $\ell(c)>0$, this implies $\s_0<0$ and thus $h_0$ is increasing. Moreover, $\dot w(a)<0$, which concludes the proof. 
\end{proof}
We summarize our results in the following theorem:

\begin{theorem}\label{thm:ystarH0}
	Let $\eta$ be defined by \eqref{def:eta}. Then we have:
	\begin{enumerate}[label=$\bullet$]
		\item[(i)]  $h_0$ is increasing,
		\item[(ii)]  $h_0$ is increasing,
		\item[(iii)] $h_0$ is decreasing,
		\item[(iv)]  if $c_-<\eta$, then $h_0$ is decreasing, if $c_->\eta$, then $h_0$ is increasing. 
	\end{enumerate}
In all cases $\a$ is decreasing.
\end{theorem}
\begin{remark}\label{rem:slitheights}
We have seen in the proof of Lemma \ref{lem:Remez} that $w_\e(x_0)$ increased monotonically, if some other slit height was decreased. Theorem \ref{thm:ystarH0}, in particular  case (iii) and case (iv) show that such a monotonicity is lacking for Akhiezer polynomials, which makes the situation essentially different to the Remez extremal problem. 
\end{remark}
\section{Reduction to Akhiezer polynomials}\label{sec:Reduction}
The goal of this section is to finish the proof of Theorem \ref{thm:AkhRemez}. That is, if $x_0$ is in an internal gap, then the extremal polynomial is an Akhiezer polynomial. Recall that  in contrast to Section \ref{sec:Remez} now it is possible that there is an extension outside of $ [-1,1]$, moreover, this is a generic position. All possible types of extremizer were described in Corollary \ref{cor:AkhPoly} and the discussion above the corollary. Thus, it remains to show that on $[-1,1]$ the extremal configuration in fact only has one gap, i.e., the extremal set is of the form $E(\a,\d)$ for some $\a\in (-1+\d,1-\d)$.

First we will show that $E$ has at most two gaps on $[-1,1]$. Let $T_n(z,E)$ denote the extremizer of \eqref{eq:upperEnvSet} for the set $E$.
\begin{lemma}\label{lem:SetOutside}
	Let $E\subset [-1,1]$ and let $x_0$ be in an internal gap of $E$. Then, there exists a two-gap set $\tilde{E}$, such that 
\begin{align}\label{eq:1June10}
	T_n(x_0,E)=T_n(x_0,\tilde E),\quad |E|=|\tilde E|.
\end{align}
\end{lemma}
\begin{proof}
We will only prove the case that there are no boundary gaps. Moreover, let us assume that $E$ is already maximal, i.e., $T_n^{-1}([-1,1],E)\cap[-1,1]=E$. Let us write
\begin{align*}
E=[-1,1]\setminus\bigcup_{j=1}^g(a_j,b_j),
\end{align*}
and let us denote the gap which contains $x_0$ by $(a,b)$. Let $\Pi$ be the comb related to $E$ and let us assume that the slit corresponding to $(a,b)$ is placed at $0$ and let $c$ denote the critical point of $T_n$ in $(a,b)$. Assume that there are two additional gaps $(a_1,b_1)$ and $(a_2,b_2)$ with slit heights $h_k$ and critical points $c_k$, $k=1,2$. We will vary the slit heights $h_k$ and $h$ such that $w_\e(x_0)=x_0$ and $|E_\e|=|E|$. Therefore, we get
	\begin{align}\label{1jul1}
	\dot w(z)=\frac{z^ 2-1}{c_1^ 2-1}\frac{\s_{1}}{c_{1}-z}+\frac{z^ 2-1}{c_2^ 2-1}\frac{\s_{2}}{c_{1}-z}+\frac{z^ 2-1}{c^ 2-1}\frac{\s}{c-z}.
	\end{align}
	The constraints
	\begin{align}\label{eq:constraintsGeneral}
	\sum_{j=1}^{g}(\dot w(b_j)-\dot w(a_j))=0, \quad \dot w(x_0)=0
	\end{align}
	will guarantee that \eqref{eq:1June10} is satisfied. Our goal is to find $\s_1>0$ and $\s,\s_2$, such that \eqref{eq:constraintsGeneral} is satisfied. 	Let us define 
	\begin{align*}
	H(z)=\frac{(z^2-1)(z-x_0)}{(z-c_1)(z-c_2)(z-c)}.
	\end{align*}
	Due to the second constraint in \eqref{eq:constraintsGeneral}, we have 
	\begin{align}\label{eq:2Feb1}
	\dot w(z)=f(z)H(z),
	\end{align}
	where $f(z)$ is linear. Thus $f(z)=K(z-\xi)$ or $f(z)=K$. We need to check that we can always find $\xi$ such that the first constraint in \eqref{eq:constraintsGeneral} is satisfied. If 
	$
		\sum_{j=1}^{g}(H(b_j)-H(a_j))=0,
	$
	we set $f(z)=K$. If 
	$
	\sum_{j=1}^{g}(H(b_j)-H(a_j))\neq 0,
	$
	we define 
	\begin{align*}
		\xi=\frac{\sum_{j=1}^{g}(b_jH(b_j)-a_jH(a_j))}{\sum_{j=1}^{g}(H(b_j)-H(a_j))}.
	\end{align*}
		In any case, we define $\dot w$ by \eqref{eq:2Feb1} and  $\s_{1}$, $\s_2$, $\s$ as the residues of this function at $c_1, c_2, c$, respectively. Now we have to distinguish two cases. If $\xi\neq c_1$, we can choose $K$ so that $\s_1>0$ and decrease $h_1$. If at some point $\xi=c_1$, we can choose $K$ so that $\s_{2}$ will be decreased. Note that in this case $h_1$ remains unchanged. In particular, it won't increase again. Hence, this procedure allows us to ``erase'' all but one additional gap. 
	The case of boundary gaps works essentially in the same way, only using instead of the variations used above, variations as described in Lemma \ref{lem:variationTranslation}. 
\end{proof}

\begin{proof}[Ending Proof of Theorem \ref{thm:AkhRemez}] First, assume that the extremizer is in a generic position, that is there is an extra interval $I_n\subset \bbR\setminus [-1,1]$. We have
$T_n(z,E)$ with
\begin{equation}\label{1jul2}
E=[-1,1]\setminus\left((a,b)\cup(a_1,b_1)\right)
\end{equation}
and $|T_n(z,E)|\le 1$ for $z\in E\cup I_n$. The corresponding comb has three slits, which heights  we denote by $h_{out}, h, h_1$. Our goal is to show that we can reduce the value  $h_1$. Varying all three values, we get that the corresponding infinitesimal variation as described by expression \eqref{1jul1}, with the parameters $\sigma_{out},\s, \s_1$. Therefore, we still can satisfy the two constraints in \eqref{eq:constraintsGeneral}, and choose one of the parameters positive. Since the direction of the variation of the heights $h_{out}$ and $h$ is not essential for us, we choose $\s_1>0$, and therefore reduce the size of $h_1$.

In a degenerated case we can use variations, which were described in Section \ref{sec:Akhiezer}. Assume that $E$ is of the form \eqref{1jul2}, but the corresponding comb has only two non-trivial teeth of the heights $h$ and $h_1$. WLOG we assume that $b<a_1$. We have
$$
T_n(z,E)=\cos\theta(z), \quad \theta(\pi n_-)=-1,\quad \theta(\pi n_+)=1.
$$
We will apply the third variation, see Figure \ref{fig:cases3and4}, left. That is, we will reduce the value $h_1$ and move the preimage of $-1$ in the positive direction. According to 
\eqref{1jul4}, see also \eqref{1jul5}, we obtain
$$
\dot w(z)=-\tau(z-1)+\frac{1-z^2}{1-c_1^2}\frac{\s_1}{c_1-z}, \quad \tau<0,\quad \s_1>0.
$$
Let us point out that with an arbitrary choice of the parameter $\s_1>0$ and $\tau^-<0$ we get an increasing function. Therefore,
$$
\dot w(1)-\dot w(b_1)>0,\ \dot w(a_1)-\dot w(b)>0,\ \dot w(a)-\dot w(-1)>0.
$$
Thus with a small variation $\e$ of this kind we get 
$$
E_\e=[w_\e(-1),w_\e(1)]\setminus\left((w_\e(a),w_\e(b)\cup(w_\e(a_1),w_\e(b_1)\right)
$$
with
	\begin{align*}
|E_\e|=&(w_\e(1)-w_\e(b_1))+( w_\e(a_1)-w_\e(b))+ (w_\e(a)-w_\e(-1))\\
>&(1-b_1)+ (a_1-b)+(a-(-1))=|E|.
	\end{align*}

On the other hand $\dot w(z)$ has a trivial zero $z=1$ and a second one, which we denote by $y_*$. Since
$$
\tau+\frac{1+y_*}{1-c_1^2}\frac{\s_1}{c_1-y_*}=0,
$$
we get
$$
y_*=y_*(\rho_1,\rho_2)=\rho_1(-1)+\rho_2 c_1,
$$
where
$$
\rho_2 =\frac{-\tau }{-\tau+\frac{\s_1}{1-c^2_1}},\quad \rho_1=1-\rho_2,\quad \rho_{1,2}>0.
$$
Thus with different values of $\tau<0$ and $\s_1>0$, we can get  an arbitrary value $y_*\in(-1,c_1)$.

Assume that $x_0<c$. We choose $\rho_1,\rho_2$ such that $y_*>c$ (recall our assumption $b<a_1$, therefore this is possible). Then $\dot w(x_0)<0$. For a small $\e$ we get
$w_\e(x_0)<x_0$. By definition $T_n(w_\e(x_0),E_\e)=T_n(x_0,E)$. Since in this range $T_n(w_\e(x_0),E_\e)$ is increasing, we get 
$T_n(x_0,E_\e)>T_n(x_0,E)$, that is $T_n(z,E)$ was not an extremal polynomial for the Andievskii problem.

If $x_0\in (c,b)$ we choose $y_*<c$. Then $\dot w(x_0)>0$. We repeat the same arguments, having in mind that in this range $T_n(z,E)$ is decreasing. Thus, we arrive to the same contradiction $T_n(x_0,E_\e)>T_n(x_0,E)$.
\end{proof}

\section{Asymptotics}\label{sec:Asymptotics}
The goal of this section is to prove Theorem \ref{thm:TotikWidomBounds} and to give a description of the upper envelope \eqref{eq:upperEnvelopeGreen}. 
\subsection{Totik-Widom bounds}
We need an asymptotic result for Akhiezer polynomials. Recall that $E(\a,\d)=[-1,1]\setminus(\a-\d,\a+\d)$ and let $A_{n,\a,\d}$ denote the associated Akhiezer polynomial. Moreover, we denote $\hat E_n(\a,\d)=A_{n,\a,\d}^{-1}([-1,1])=E(\a,\d)\cup I_n$ and $\hat \O_n=\overline{\bbC}\setminus\hat E_n$. We have described the shape of  the additional interval $I_n$ in Theorem \ref{thm:extension} and the discussion below. The following Lemma is known in a much more general context \cite[Proposition 4.4.]{ChriSiZiYuDuke}.
\begin{lemma}
	Let $\a,\d$ be fixed and  $A_{n,\a,\d}$ be the associated Akhiezer polynomial. Let $x_n\in I_n$ denote the single zero of $A_{n,\a,\d}$ outside of $[-1,1]$. For any $y\in (\a-\d,\a+\d)$
	\begin{align}\label{eq:AkhiezerRep}
	\log 2|A_{n,\a,\d}(y)|=nG(y,\infty,\hat\Omega_n)+o(1),
	\end{align} 
If we pass to a subsequence such that $\lim\limits_{k}x_{n_k}=x_\infty\in(\bbR\cup\{\infty\})\setminus(-1,1)$, then  
	\begin{align}\label{eq:GreenAsymptotic}
	\lim\limits_{k\to\infty}n_k(G(y,\infty;\Omega(\a,\d))-G(y,\infty;\hat\Omega_{n_k}))=G(y,x_\infty,\Omega(\a,\d)),
	\end{align} 
	where $\Omega(\a,\d)=\overline{\bbC}\setminus E(\a,\d)$. 
\end{lemma}

\begin{proof}[{Proof of Theorem \ref{thm:TotikWidomBounds}}]
We start with the case that the sup in \eqref{eq:upperEnvelopeGreen} is attained at some internal point $\a_0\in (-1+\d,0]$. Let $T_{n,\d,x_0}$ be the extremizer of \eqref{eq:upperEnvelope} and set $\hat E_n=T_{n,\d,x_0}^{-1}([-1,1])$ and $\hat \Omega_n=\overline{\bbC}\setminus \hat E_n$.  Due to Theorem \ref{thm:AkhRemez}, $\hat E_n$ is either $[-1+2\d,1]$ or $\hat E_n=E(\a_n,\d)\cup I_n$ for some $\a_n$ and some additional interval outside of $[-1,1]$. In the following, we will denote $G(x,\infty,\Omega)=G(x,\Omega)$ and we note that by definition $G(x,\Omega(\a,\d))=G_{\a,\d}(x)$. Put $E_n=\hat E_n\cap [-1,1]$ and $\Omega_n=\overline{\bbC}\setminus E_n$. Due to extremality of $\a_0$, we have
\begin{align*}
G(x_0,\Omega(\a_0,\d))\geq G(x_0,\Omega_n).
\end{align*}
  Since $\hat{\O}_n\subset\Omega_n$, we get 
\begin{align*}
G(x_0,\Omega_n)\geq G(x_0,\hat \Omega_n).
\end{align*}
Using \eqref{eq:AkhiezerRep} or  \eqref{eq:RemezRep} and the extremality property of $T_{n,\d,x_0}(x_0)$, we get
 \begin{align*}
 nG(x_0,\hat \Omega_n)&=\log 2|T_{n,\d,x_0}(x_0)|+o(1)\\
 &\geq \log 2|T_{n}(x_0,E(\a_0,\d))|+o(1).
 \end{align*}
Equation \eqref{eq:GreenAsymptotic} yields 
 \begin{align*}
  \log 2|T_n(x_0,E(\a_0,\d))|\geq nG(x_0,\Omega(\a_0,\d))-C+o(1),
 \end{align*}
 where 
 \begin{align}\label{eq:WidomConstant}
 C=\sup_{x\in(\bbR\cup\{\infty\})} G(x,x_0,\Omega(\a_0,\d)).
 \end{align}
By definition $|T_{n,\d,x_0}(x_0)|=L_n(x_0)$ and therefore combining all inequalities finishes the proof of \eqref{eq:TotikWidomAkh}. The proof for the boundary case is essentially the same. Only in the last step, due to representation \eqref{eq:RemezRep}, there is no extra constant C (due to the fact that there is no extension of the domain) and we get \eqref{eq:TotikWidomRem}.
\end{proof}

\subsection{The asymptotic diagram}\label{sec:AsympDiagram}
In this section we introduce an asymptotic diagram, which provides a  description of the upper envelope $\Phi_\d(x)$.  First of all we prove Lemma \ref{lem:UpperEnvelopeGreen}. 
\begin{proof}[Proof of Lemma \ref{lem:UpperEnvelopeGreen}]
	Recall the explicit representation of Green functions for two intervals as elliptic integrals. For $a=\a-\d, b=\a+\d$ we have
	\begin{align}\label{eq:GreenElliptic}
	G_{\a,\d}(x)=\int_a^x\frac{(c-\xi)d\xi}{\sqrt{(\xi^2-1)(\xi-a)(\xi-b)}}, 
	\end{align}
	where
	\begin{align}\label{eq:CriticalPoint}
	c=c(\a)=\frac{\int_a^b\frac{\xi d\xi}{\sqrt{(\xi^2-1)(\xi-a)(\xi-b)}}}{\int_a^b\frac{d\xi}{\sqrt{(\xi^2-1)(\xi-a)(\xi-b)}}}.
	\end{align}

	If for fixed $x_0$ the sup is attained at an interval point, then clearly 
	\begin{align}\label{eq:partialGreen}
	\pd_\a G_{\a,\d}(x)=0
	\end{align}
	holds. Thus, it remains to show that \eqref{eq:partialGreen} has a unique solution $x_0(\a)$. Due to \eqref{eq:GreenElliptic} we have 
	$$
	\pd_\a\log \pd_xG_{\a,\d}(x)=\frac{\dot c}{c-x}-\frac 1 2\frac{\dot a}{a-x}-\frac 1 2\frac{\dot b}{b-x}
	$$
	Since $\dot a=\dot b=1$ we get
	\begin{equation}\label{24jun1}
	\pd_x  \pd_\a G_{\a,\d}(x)=\left\{\dot c+\frac{(c-x)(x-\a)}{(x-a)(x-b)}\right\}\frac 1{\sqrt{(x^2-1)(x-a)(x-b)}}.
	\end{equation}
	Note that
	$$
	\inf_{x\in(a,b)}\frac{(c-x)(x-\a)}{(x-a)(x-b)}>-1.
	$$
	Since  the distance $|c(\a)-b(\a)|$ monotonically increases with $|\a|$, we have $\dot c(\a)>1$. Therefore, we get $\pd_x\pd_\a G_{\a,\d}(x)>0$ in \eqref{24jun1}. That is, $\pd_\a G_{\a,\d}(x)$ is increasing for $x\in (a,b)$. Moreover, $\lim_{x\to a} \pd_\a G_{\a,\d}(x)=-\infty$ and  $\lim_{x\to b} \pd_\a G_{\a,\d}(x)=+\infty$ and we obtain that a zero $x_0(\a)$
	of the function $\pd_\a G_{\a,\d}(x)$ in $(a,b)$
	exists and is unique. 
\end{proof}

Thus, the limiting behavior of $L_{n,\d}(x_0)$, $n\to\infty$, can be distinguished by a diagram with two curves, which we describe in the following proposition, see also Example \ref{ex1}.

\begin{proposition}\label{prop:UpperEnvelope} In the range  $x\in[-1,0]$,
	$\Phi_\d(x)$ represents the upper envelope of the following two curves.
	The first one is given explicitly
	\begin{equation}\label{22jun1}
	y=G_\d(x)=\log \left(\frac{\d-x}{1-\d}+\sqrt{\left(\frac{\d-x}{1-\d}\right)^2-1}\right),
	\end{equation}
	and the second one is given in parametric form
	\begin{equation}\label{22jun2}
	x=x_0(\a),\quad y=G_{\a,\d}(x_0(\a)), \quad \a\in (-1+\d,0].
	\end{equation}
	Moreover, the end points of the last curve are given explicitly by
	\begin{equation}\label{22jun3}
	\left(0,\frac 1 2\log\frac{1+\d}{1-\d}\right) \ \text{for}\ \a=0,\quad (-1+2\d,0)\ \text{for}\ \a\to -1+\d.
	\end{equation}
\end{proposition}

\begin{proof}
	According to Lemma \ref{lem:UpperEnvelopeGreen}, if $\Phi_\d(x)$ is assumed at the end point $\a\to-1+\delta$, then it is the Green function in the complement to the interval $[-1+2\d,1]$, which has a well known representation \eqref{22jun1}. Alternatively, $x=x_0(\a)$ and $\Phi_\d(x)=G_{\a,\d}(x_0(\a))$ for a certain $\a$ in the range, what is \eqref{22jun2}.
	
	Further, for $\a=0$, $G_{0,\d}(x)$ is the Green function related to  two symmetric intervals. Due to the symmetry $x_0(0)=0$ and 
	$G_{0,\d}(x)$ can be reduced to the Green function of a single interval $[\d^2,1]$. We get
	$$
	G_{0,\d}(0)=\frac 1 2\log\frac{1+\d}{1-\d}.
	$$
	
	Thus, it remains to prove the last statement of the proposition. Our proof is based on expressing $\pd_\a G_{\a,\d}$ in terms of elliptic integrals and manipulating those.
	It will be convenient to make a standard substitution in \eqref{eq:GreenElliptic}. Let $\xi(\psi,\a)=\a-\d\cos\psi$, then
	\begin{align}\label{eq:2June18}
	G_{\a,\d}(x)=\int_{\phi(x,\a)}^\pi\frac{\xi(\psi,\a)-c(\a)}{\sqrt{1-\xi(\psi,\a)^2}}d\psi,\quad \phi(x,\a)= \arccos\frac{\a-x}{\d}.
	\end{align}
	Differentiating \eqref{eq:2June18} we get
	\begin{align}\label{24feb1}
	\pd_\a G_{\a,\d}(x)=I_1(x,\a)+I_2(x,\a)-\dot c(\a) I_3(x,\a)
	\end{align}
	where
	\begin{align*}
	I_1(x,\a)=&\int_{\phi(x,\a)}^\pi\frac{1- c(\a)\xi(\psi,\a)}{(1-\xi(\psi,\a)^2)\sqrt{1-\xi(\psi,\a)^2}}d\psi
	\\
	I_2(x,\a)=&\frac{x-c(\a)}{\sqrt{(1-x^2)(\d^2-(x-\a)^2)}},
	\quad I_3(x,\a)=\int_{\phi(x,\a)}^\pi\frac{d\psi}{\sqrt{1-\xi(\psi,\a)^2}}.
	\end{align*}
	Let  $\e=\e(\a)$ be such that $\e\to 0$ as $\a\to  -1+\delta$, to be chosen later on, and   let 
	\begin{align}\label{eq:2April1}
	x(\a)=\a-\d\cos(\pi-\e).
	\end{align} 
	Direct estimations  show that 
	\begin{equation*}
	I_1(x(\a),\a)=o(1),\quad I_2(x(\a),\a)=\frac{a_1(\a)}{\e}+o(1),
	\end{equation*}
	where
	$\lim_{\a\to -1+\d}a_1(\a)>0$. The integral $I_3(x(\a),\a)$ also tends to zero, but we will need  a more precise decomposition
	\begin{equation}\label{22jun9}
	I_3(x(\a),\a)=b_1(\a)\e+\frac{b_2(\a)}{3}\e^3+O(\e^5)
	\end{equation}
	with $\lim_{\a\to -1+\d}b_1(\a)>0$.  
	Indeed, we have
	\begin{align*}
	I_3= \int_{\phi(x(a),\a)}^\pi&\frac{d\psi}{\sqrt{1-\xi(\psi,\a)^2}}=\int_{0}^\e\frac{d\psi}{\sqrt{1-(\a+\d\cos\psi)^2}}.
	\end{align*} 
	Therefore, for $\psi$ sufficiently small we can use the expansion 
	$$
	\frac{1}{\sqrt{1-(\a+\d\cos\psi)^2}}=b_1(\a)+b_2(\a)\psi^2+O(\psi^4),
	$$ 
	where $O(\psi^4)$ is uniform in $\a$. We get \eqref{22jun9}. In the same time, 
	$$
	\lim_{\a\to -1+\d}b_1(\a)=\frac{1}{\sqrt{1-(-1+2\d)^2}}>0.
	$$

	Collecting all three terms, we obtain
	\begin{align}
	\pd_\a G_{\a,\d}(x(\a))&=\frac{a_1(\a)}{\e}-\dot c(\a)\left(b_1(\a)\e+\frac{b_2(\a)}{3}\e^3\right)+o(1) \nonumber\\
	&=\k_1(\a)\left(\frac{\k(\a)}{\e}-\dot c(\a)\left(\frac{\e}{\k(\a)}+\frac{\k_2(\a)}{3}\e^3\right)\right)+o(1),\label{22jun5}
	\end{align}
	where $\k_1(\a)=\sqrt{a_1(\a)b_1(\a)}$, $\k(\a)=\sqrt{a_1(\a)/b_1(\a)}$ and $\k_2(\a)$ is chosen appropriately. Recall that $\k_1(\a)$ and $\k(\a)$ have positive and finite limits as $\a\to -1+\d$. 
	
	Similar manipulations with elliptic integrals show that
	$\dot c(\a)\to +\infty$ as $\a\to-1+\d$. In fact,   the rate of this divergence is (see Appendix)
	\begin{align*}
	\dot c(\a)\simeq (\varepsilon\log \varepsilon)^{-2}, \quad \a=-1+\d(1+\frac 1 2 \varepsilon^2),\  \varepsilon\to 0,
	\end{align*}
	but such accuracy is not needed  for our purpose. 
	
	Having these estimations, we can find 
	a suitable interval $[x_-(\a),x_+(\a)]$, with the limit
	\begin{equation}\label{22jun11}
	\lim_{\a\to -1+\d}x_\pm(\a)= -1+2\d,
	\end{equation}
	such that
	$\pd_\a G_{\a,\d}(x)$ changes its sign in it and therefore this interval contains the unique solution of the equation $\pd_\a G_{\a,\d}(x)=0$.
	
	Define  
	\begin{align*}
	\e_\pm(\a)=\frac{\k(\a)}{\sqrt{\dot c(\a       )}\pm 1}.
	\end{align*}
	Since  $\dot c(\a)\e_\pm(\a)^3=o(1)$, \eqref{22jun5} gets the form
	\begin{align}\nonumber
		\k_1(\a)^{-1}\pd_{\a}G(x_\pm(\a),\a)&=\sqrt{\dot c(\a)}\pm 1-\frac{\dot c(\a)}{\sqrt{\dot c(\a)}\pm 1}+o(1)\\
		&=\frac{\pm 2\sqrt{\dot c(\a)}+1}{\sqrt{\dot c(\a)}\pm 1}+o(1),\label{eq:July3}
	\end{align}
	where $x_\pm(\a)$ is defined by \eqref{eq:2April1} for $\e=\e_\pm(\a)$. For $\dot c(\a)$ sufficiently large, we obtain  in \eqref{eq:July3} both positive and negative values, and simultaneously  
	we have  \eqref{22jun11}. Consequently,  $x_0(\a)\to -1+2\d$. 
	
\end{proof}

\begin{corollary}\label{cor43}
	Let $\d_*$ be a unique solution of the equation
	$$
	\d_*^2=\frac{1-\d_*}{1+\d_*},\quad \d_*\in(0,1),
	$$
	numerically $\d_*=0.543689...$. Then for $\d<\d_*$ $\Phi_{\d}(x)$ does not coincide identically with $G_\d(x)$ (in  its range $x\in (-1,0]$). 
	
	On the other hand, for an arbitrary $\d>0$ there exists $x_*(\d)>-1$ such that  $\Phi_{\d}(x)$ and $G_\d(x)$ coincide in
	$[-1,x_*(\d))$.
\end{corollary}

\begin{proof}
	The first claim follows by a direct comparison of $G_\d(0)$ and $G_{0,\d}(0)$, see \eqref{22jun1} and \eqref{22jun3}.
	
	For a fixed $\d$ we define 
	\begin{equation}\label{22jun7}
	x_*(\d)=\inf_{\a\in(-1+\d,0]} x_0(\a)
	\end{equation}
	By \eqref{22jun3} and continuity $x_*(\d)>-1$. Thus the curve \eqref{22jun2} does not intersect the range $[-1,x_*(\d))$.
\end{proof}

\begin{remark}
	We do not claim here that $[-1,x_*(\d))$ with $x_*(\d)$ given by \eqref{22jun7} is the biggest possible interval on which $\Phi_\d(x)=G_\d(x)$, see Example \ref{ex1} bellow for details.
\end{remark}

\begin{example}\label{ex1} A numerical example of the  asymptotic diagram for $\d=0.4$ is given in Figure \ref{fig:Regions} (diagrams for other values of $\d<0.5$ look pretty the same).
	Let $x_s(\d)$ be the switching point between two (Remez and Akhiezer) extremal configurations, i.e.,
	\begin{figure}
		\begin{center}
			\includegraphics[scale=0.5]{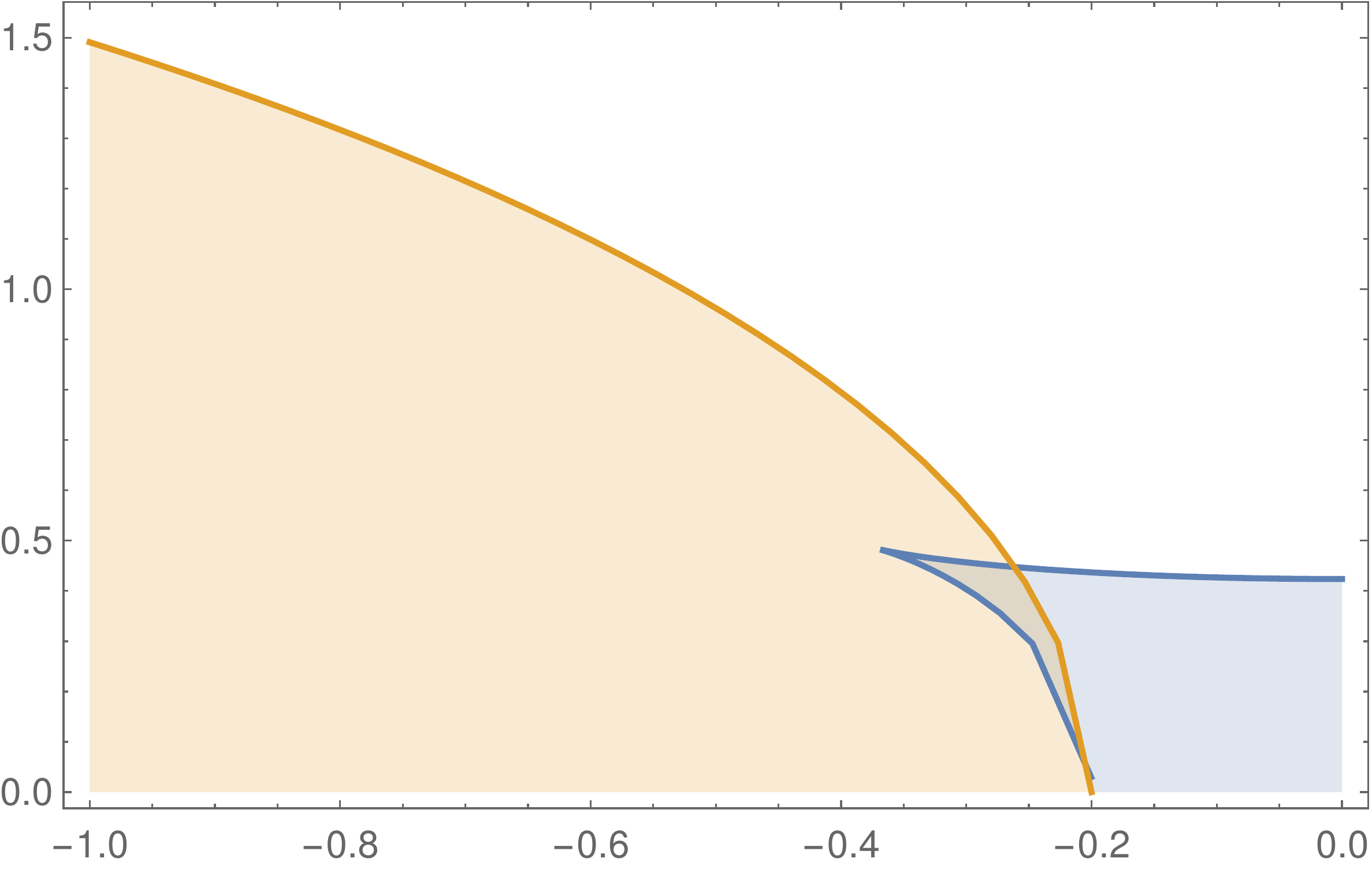}
			\caption{The asymptotic diagram for $\d=0.4$.}
			\label{fig:Regions}
		\end{center}	
	\end{figure}
\end{example}
$$
\Phi_\d(x_s)=G_\d(x_s(\d))=G_{\a_s,\d}(x_0(\a_s)),\quad x_0(\a_s)=x_s(\d).
$$
Recall that $x_*(\d)$ was defined in \eqref{22jun7}.
On the diagram we can observe the following four regions:
$(-1, x_*(\d)), (x_*(\d),x_s(\d)), (x_s(\d),-1+2\d)$ and $(-1+2\d,0)$. Note that we discuss the case $-1+2\d<0$, i.e., $\d<0.5<\d_*$.

\begin{itemize}
	\item[a)] $x\in (-1+2\d,0)$. In this case $x\in (\a-\d,\a+\d)$ implies that such an interval  is a subset of $(-1,1)$ even in the leftmost position $\a=-\d+x$. Therefore, the function $G_{\a,\d}(x)$ for a fixed $x$ and $\a\in (x-\d,x+\d)$ attains its maximum at some internal point and we get the case $\pd_\a G_{\a,\d}(x)=0$.
	\item[b)] $x\in(x_s(\d),-1+2\d)$. As soon as $x<-1+2\d$ the left boundary for a possible value of $\a$ is given by $\a-\d=-1$. Respectively the supremum of $G_{\a,\d}(x)$ for a fixed $x$ can be attained either at an internal point $\a\in(-1+\d,x+\d)$ or as the limit at the left end point. In this range it is still attained at an internal point. Note that besides the local maximum the function gets a local minimum (the second branch of  the curve \eqref{22jun2} with the same coordinate $x_0(\a)=x$).
	\item[c)] $x\in (x_*(\d),x_s(\d))$. For such $x$ the function $G_{\a,\d}(x)$  has still  its local maximum and minimum, but the biggest value is attained at the boundary point $\a=-1+\d$, i.e., $\Phi_\d(x)=G_\d(x)$.
	\item[d)] $x\in(-1,x_*(\d))$. At  $x=x_*(\d)$ the points of local maximum and minimum  of  the function $G_{\a,\d}(x)$ collide, that is, in fact, they produce  an inflection point. The function $G_{\a,\d}(x)$ become monotonic in this region. Its supremum is  the limit at the boundary point $\a=-1+\d$, see the second claim in Corollary \ref{cor43}.
\end{itemize}

\appendix

\section{Lemma on the limit of $\dot c(\a)$}

\begin{lemma}\label{lem:cdotDiverges}
	Set $\a=-1+\d(1+\frac 1 2 \e^2)$. Then $\dot c(\a)$ tends to $+\infty$ as $\e\to 0$ with the rate 
	\begin{align*}
	\dot c(\a)\simeq(\e\log\e)^{-2}.
	\end{align*}
\end{lemma} 
\begin{proof}
	By \eqref{eq:CriticalPoint}, we have 
	\begin{align*}
	\a-c=\frac{\int_a^b\frac{(\a-\xi)d\xi}{\sqrt{(\d^2-(\xi-\a)^2)(1-\xi^2)}}}{\int_a^b\frac{d\xi}{\sqrt{(\d^2-(\xi-\a)^2)(1-\xi^2)}}}
	\end{align*}
	Making the change of variables 
	\begin{align*}
	\xi=\xi(\a,\phi)=\a-\d\cos\phi,\quad \phi\in(0,\pi)
	\end{align*}
	we get
	\begin{align*}
	\a-c=\frac{\int_0^\pi\frac{\d^2\cos\phi \sin\phi d\phi}{\d\sin\phi\sqrt{1-x^2}}}{\int_0^\pi\frac{\d\sin\phi d\phi}{\d\sin\phi\sqrt{1-x^2}}}=\d 
	\frac{\int_0^\pi\frac{\cos\phi  d\phi}{\sqrt{1-x^2}}}{\int_0^\pi\frac{d\phi}{\sqrt{1-x^2}}}
	\end{align*}
	Thus, introducing 
	\begin{align*}
	u(\a)=\int_0^\pi\frac{\cos\phi  d\phi}{\sqrt{1-\xi^2}},\quad v(\a)=\int_0^\pi\frac{ d\phi}{\sqrt{1-\xi^2}},
	\end{align*}
	we have
	\begin{align*}
	1-\dot c=\frac{\d}{v^2}\det\left[\begin{matrix} \dot u&\dot v\\ u& v\end{matrix}\right].
	\end{align*}
	Since $\pd_\a \xi(\phi,\a)=1$, we have
	\begin{align*}
	\dot u(\a)=\int_0^\pi\frac{\xi\cos\phi  d\phi}{(1-\xi^2)\sqrt{1-\xi^2}},\quad \dot v(\a)=\int_0^\pi\frac{ \xi d\phi}{(1-\xi^2)\sqrt{1-\xi^2}}.
	\end{align*}
	Using the definition of $\xi(\a,\phi)$ we get
	\begin{align*}
	\d\det\left[\begin{matrix} \dot u&\dot v\\ u& v\end{matrix}\right]&=-\det\left(\left[\begin{matrix} \dot u&\dot v\\ u& v\end{matrix}\right]
	\left[\begin{matrix} -\d&0\\ \a+1 & 1\end{matrix}\right]\right)\\
	&=
	-\det\left[\begin{matrix} \int_0^\pi\frac{\xi(1+\xi)d\phi}{(1-\xi^2)\sqrt{1-\xi^2}}&\int_0^\pi\frac{ \xi d\phi}{(1-\xi^2)\sqrt{1-\xi^2}}\\
	\int_0^\pi\frac{(1+\xi)d\phi}{\sqrt{1-\xi^2}}&\int_0^\pi\frac{  d\phi}{\sqrt{1-\xi^2}}
	\end{matrix}\right].
	\end{align*}
	Finally
	\begin{align}\label{eq:1June17}
	\dot c=1+\frac{1}{v^2}\det\left[\begin{matrix} \int_0^\pi\frac{\xi d\phi}{(1-\xi)\sqrt{1-\xi^2}}&\int_0^\pi\frac{ \xi d\phi}{(1-\xi^2)\sqrt{1-\xi^2}}\\
	\int_0^\pi\sqrt{\frac{1+\xi}{1-\xi}}d\phi&\int_0^\pi\frac{  d\phi}{\sqrt{1-\xi^2}}
	\end{matrix}\right]=1+\frac{1}{v^2}\det\left[\begin{matrix}I_1&I_2\\
	I_3&v
	\end{matrix}\right].
	\end{align}
	Now we insert $\a=-1+\d(1+\frac 1 2 \e^2)$, $\e\to 0$. For a sufficiently small $\phi_0$, we have
	\begin{align*}
	1+\xi=1+\a-\d\cos\phi=\frac{\d} 2(\phi^2+\e^2)+O(\phi^4),\quad \phi\in(0,\phi_0).
	\end{align*}
	Recall that $\xi(\pi,\a)=b$. Therefore, the following limit 
	\begin{align*}
	\lim_{\e\to 0}\int_{\phi_0}^\pi\frac{d\phi}{\sqrt{1-\xi^2}}
	\end{align*}
	exists. Thus, we have 
	\begin{align*}
	v=\int_0^\pi\frac{d\phi}{\sqrt{1-\xi^2}}\simeq\int_{0}^{\phi_0}\frac{d\phi}{\sqrt{1+\xi}}\simeq\int_{0}^{\phi_0}\frac{d\phi}{\sqrt{\d(\e^2+\phi^2)}}\simeq\int_{0}^{\phi_0/\e}\frac{dt}{\sqrt{1+t^2}}
	\simeq -\log \e.
	\end{align*}
	Also
	\begin{align*}
	I_1=\int_0^\pi\frac{\xi d\phi}{(1-\xi)\sqrt{1-\xi^2}}\simeq\int_{0}^{\phi_0}\frac{d\phi}{\sqrt{1+\xi}}\simeq -\log\e.
	\end{align*}
	Moreover, for $I_3$ we get
	\begin{align*}
	\lim_{\e\to 0}\int_0^\pi\sqrt{\frac{1+\xi}{1-\xi}}d\phi=\int_0^\pi\sqrt{\frac{\d(1-\cos\phi)}{2-\d+\d\cos\phi}}d\phi >0.
	\end{align*}
	As before, we can split up $I_2$ and get 
	\begin{align*}
	I_2\simeq\int_0^{\phi_0}\frac{ \xi d\phi}{(1-\xi^2)\sqrt{1-\xi^2}}&\simeq -\int_0^{\phi_0}
	\frac{  d\phi}{((\phi^2+\e^2))^{3/2}}\\
	&=-\frac 1{\e^2}\int_0^{\phi_0/\e}\frac{dt}{(1+t^2)^{3/2}}\\
	&\simeq-\frac {1}{\e^2}.
	\end{align*}
	Collecting all terms and inserting it into \eqref{eq:1June17} yield the claim. 
\end{proof}

\providecommand{\MR}[1]{}
\providecommand{\bysame}{\leavevmode\hbox to3em{\hrulefill}\thinspace}
\providecommand{\MR}{\relax\ifhmode\unskip\space\fi MR }
\providecommand{\MRhref}[2]{%
	\href{http://www.ams.org/mathscinet-getitem?mr=#1}{#2}
}
\providecommand{\href}[2]{#2}


\begin{thebibliography}{10}
	
	\bibitem{Akh1} N. Achyeser [N.I. Akhiezer], \emph{\"Uber einige Funktionen, die in gegebenen Intervallen am wenigsten von Null abweichen}, Izv. Kazan, Fiz.-Mat. Obshch. (3) \textbf{3} (1928), 1--69.
	
	\bibitem{Akh2} \bysame, \emph{\"Uber einige Funktionen, welche in zwei gegebenen Intervallen am wenigsten von Null abweichen}, I, II, III,  Izv. Akad. Nauk SSSR, \textbf{1932}, 1163-1202; \textbf{1933}, 309-344, 499-536.
	
	\bibitem{Akh3} N.~I. Akhiezer, \emph{Lectures on Approximation Theory}, 2nd rev. ed., ``Nauka", Moscow, 1965; German transl., Akademie-Verlag, Berlin, 1967; Engl transl. of 1st ed., Ungar, New York, 1956.
	
	\bibitem{AkhElliptic}
	\bysame, \emph{Elements of the theory of elliptic functions},
	Translations of Mathematical Monographs, vol.~79, American Mathematical
	Society, Providence, RI, 1990, Translated from the second Russian edition by
	H. H. McFaden. \MR{1054205}
	
	\bibitem{Anii1}
	V. Andrievskii,  \textit{Pointwise Remez-type inequalities in the unit disk}, Constr. Approx. 22 (2005), no. 3, 297--308. 
	
	\bibitem{Anii2}
	\bysame,  \textit{Local Remez--type inequalities for exponentials of a potential on a piecewise analytic arc},
	J. Anal. Math. 100 (2006), 323--336.
	
	
		\bibitem{ChriSiZiInvent}
		J.~S. Christiansen, B.~Simon, and M.~Zinchenko, \emph{Asymptotics of
			{C}hebyshev polynomials, {I}: subsets of {$\Bbb R$}}, Invent. Math.
		\textbf{208} (2017), no.~1, 217--245. \MR{3621835}
	
	
	\bibitem{ChriSiZiYuDuke}
	J.~S. Christiansen, B.~Simon, P.~Yuditskii, and M.~Zinchenko, \emph{Asymptotics
		of {C}hebyshev polynomials, {II}: {DCT} subsets of {$\Bbb R$}}, Duke Math. J.
	\textbf{168} (2019), no.~2, 325--349. \MR{3909898}
	

	
	
	\bibitem{ConwayCompAna}
	J.~B. Conway, \emph{Functions of one complex variable. {II}}, Graduate Texts in
	Mathematics, vol. 159, Springer-Verlag, New York, 1995. \MR{1344449}
	
	\bibitem{EichJAT17}
	B.~Eichinger, \emph{{Szeg{\H o}-{W}idom asymptotics of {C}hebyshev polynomials
			on circular arcs}}, J. Approx. Theory \textbf{217} (2017), 15--25.
	\MR{3628947}
	
	\bibitem{EichYu18}
	B.~Eichinger and P.~Yuditskii, \emph{The {A}hlfors problem for polynomials},
	Mat. Sb. \textbf{209} (2018), no.~3, 34--66. \MR{3769214}
	
	\bibitem{Erdelyi92}
	T.~Erd\'{e}lyi, \emph{{Remez-type inequalities on the size of generalized
			polynomials}}, J. London Math. Soc. (2) \textbf{45} (1992), no.~2, 255--264.
	\MR{1171553}
	
	\bibitem{ErdelyiLiSaff94}
	T.~Erd\'{e}lyi, X.~Li, and E.~B. Saff, \emph{{Remez- and Nikolskii-type
			inequalities for logarithmic potentials}}, SIAM J. Math. Anal. \textbf{25}
	(1994), no.~2, 365--383. \MR{1266564}
	
	\bibitem{ErY12}
	A.~Eremenko and P.~Yuditskii, \emph{Comb functions}, Recent advances in
	orthogonal polynomials, special functions, and their applications, Contemp.
	Math., vol. 578, Amer. Math. Soc., Providence, RI, 2012, pp.~99--118.
	\MR{2964141}
	
	\bibitem{GarnettMarshall}
	J.~B. Garnett and D.~E. Marshall, \emph{Harmonic measure}, New Mathematical
	Monographs, vol.~2, Cambridge University Press, Cambridge, 2008, Reprint of
	the 2005 original. \MR{2450237}
	
	\bibitem{KNT}
	S. Kalmykov, B. Nagy, V. Totik, \emph{Bernstein- and Markov-type inequalities for rational functions}, 
	Acta Mathematica 219 (2017), (1), 21--63.
	
	\bibitem{KoosisLogInt1}
	P. Koosis, \emph{The logarithmic integral. {I}}, Cambridge Studies in
	Advanced Mathematics, vol.~12, Cambridge University Press, Cambridge, 1998,
	Corrected reprint of the 1988 original. \MR{1670244}
	
	\bibitem{Landkof72}
	N.~S. Landkof, \emph{Foundations of modern potential theory}, Springer-Verlag,
	New York-Heidelberg, 1972, Translated from the Russian by A. P. Doohovskoy,
	Die Grundlehren der mathematischen Wissenschaften, Band 180. \MR{0350027}
	
	\bibitem{PomUnivalent}
	C.~Pommerenke, \emph{Univalent functions}, Vandenhoeck \& Ruprecht,
	G\"{o}ttingen, 1975, With a chapter on quadratic differentials by Gerd
	Jensen, Studia Mathematica/Mathematische Lehrb\"{u}cher, Band XXV.
	\MR{0507768}
	
	\bibitem{Ransford95}
	T.~Ransford, \emph{Potential theory in the complex plane}, London Mathematical
	Society Student Texts, vol.~28, Cambridge University Press, Cambridge, 1995.
	\MR{1334766}
	
	\bibitem{Remes36}
	E.~Remes, \emph{{Sur une propri\'et\'e extremale des polyn\^omes de
			Tschebychef}}, Commun. Inst. Sci.Math. et Mecan. \textbf{13} (1936), 93--95.
	
	\bibitem{SodYud92}
	M.~Sodin and P.~Yuditskii, \emph{Functions that deviate least from zero on
		closed subsets of the real axis}, Algebra i Analiz \textbf{4} (1992), no.~2,
	1--61. \MR{1182392}
	
	\bibitem{TikhYud}
	S. Tikhonov, P. Yuditskii,
	\emph{Sharp Remez inequality}, Constr. Approx. DOI 10.1007/s00365-019-09473-2.
	
	\bibitem{Widom} 
	H.~ Widom,
	\emph {Extremal polynomials associated with a system of curves in the
		complex plane},
	Advances in Math., no. 2, \textbf{3},
	(1969), 127--232.
	
\end{thebibliography}
\end{document}